\documentclass{article}

\usepackage{arxiv}

\usepackage[utf8]{inputenc} 
\usepackage[T1]{fontenc}    
\usepackage{hyperref}       
\usepackage{url}            
\usepackage{booktabs}       
\usepackage{amsfonts}       
\usepackage{nicefrac}       
\usepackage{microtype}      
\usepackage{lipsum}
\usepackage{graphicx}
\graphicspath{ {./images/} }

\usepackage{amsmath,amssymb,euscript,psfrag,latexsym,graphicx}
\usepackage{bbm,color,amstext,wasysym,subfig,cuted,mathtools,xcolor}
\usepackage{comment}
\usepackage{caption}
\usepackage[normalem]{ulem}
\usepackage{cite}
\graphicspath{{./},{./figures/}}
\graphicspath{{./},{./figures/}}
\usepackage{booktabs}

\newcommand{\mR}{{\mathbb R}}

\newtheorem{theorem}{Theorem}

\newtheorem{assumption}[theorem]{Assumption}
\newtheorem{example}[theorem]{Example}
\newtheorem{lemma}[theorem]{Lemma}
\newtheorem{proof}[theorem]{Proof}

\newcommand{\bff}{{\mathbf f}}

\newcommand{\bk}{{\mathbf k}}
\newcommand{\bP}{{\mathbf P}}

\newcommand{\bM}{{\mathbf M}}

\newcommand{\bx}{{\mathbf x}}

\newcommand{\brho}{{\boldsymbol \rho}}

\newcommand{\bPsi}{{\boldsymbol \Psi}}
\newcommand{\bC}{{\boldsymbol{\mathcal C}}}

\newcommand{\be}{{\mathbf e}}
\newcommand{\bX}{{\mathbf X}}
\newcommand{\bY}{{\mathbf Y}}
\newcommand{\bg}{{\mathbf g}}
\newcommand{\bA}{{\mathbf A}}
\newcommand{\bB}{{\mathbf B}}
\newcommand{\bu}{{\mathbf u}}
\newcommand{\bL}{{\mathbf L}}
\newcommand{\bd}{{\mathbf d}}
\newcommand{\bD}{{\mathbf D}}
\newcommand{\bz}{{\mathbf z}}

\newcommand{\bI}{{\mathbf I}}
\newcommand{\bR}{{\mathbf R}}

\newcommand{\bs}{{\mathbf s}}

\newcommand{\bH}{{\mathbf H}}

\newcommand{\mP}{{\mathbb P}}

\newcommand{\mU}{{\mathbb U}}

\newcommand{\bc}{{\mathbf c}}
\newcommand{\chen}[1]{\textcolor{cyan}{#1}}

\newtheorem{definition}{Definition}
\newtheorem{remark}{Remark}
\title{Data-Driven Optimal Control via Linear Transfer Operators: A Convex Approach} 
\author{
Joseph Moyalan\\
Department of Mechanical Engineering\\
Clemson University\\
Clemson, SC, USA \\
\texttt{jmoyala@clemson.edu}\\
\And
Hyungjin Choi\\
Sandia  National  Laboratories\\
Albuquerque,  NM,  USA\\
\texttt{hchoi@sandia.gov}\\
\And
Yongxin Chen\\
School  of Aerospace Engineering\\
Georgia Institute of Technology\\
Atlanta, GA, USA\\
\texttt{yongchen@gatech.edu}\\
\And
Umesh Vaidya \thanks{This work was supported by the NSF under grant 2031573, 1932458, and 1942523.}\\
Department of Mechanical Engineering\\
Clemson University\\
Clemson, SC, USA \\
\texttt{uvaidya@clemson.edu}\\
}

\begin{document}

\maketitle


\begin{abstract}  
This paper is concerned with data-driven optimal control of nonlinear systems. We present a convex formulation to the optimal control problem with a discounted cost function. We consider OCP with both positive and negative discount factor. The convex approach relies on lifting nonlinear system dynamics in the space of densities using the linear Perron-Frobenius (P-F) operator. This lifting leads to an infinite-dimensional convex optimization formulation of the optimal control problem. The data-driven approximation of the optimization problem relies on the approximation of the Koopman operator using the polynomial basis function. We write the approximate finite-dimensional optimization problem as 
a polynomial optimization
which is then solved efficiently using a sum-of-squares-based optimization framework. Simulation results are presented to demonstrate the efficacy of the developed data-driven optimal control framework. 


\end{abstract}
          

\section{Introduction}
Data-driven optimal control of a nonlinear system is a problem that has significant interest with applications ranging from vehicle autonomy, robotics to manufacturing and power systems.
The traditional approach to optimal control problem (OCP) relies on solving the Hamilton Jacobi Bellman (HJB) equation \cite{fleming2012deterministic}. The HJB equation is a nonlinear partial differential equation and challenging to solve. Existing algorithms for solving HJB equations rely on iterative scheme \cite{beard1997galerkin,bertsekas2011approximate}. This iterative scheme is also at the heart of the variety of reinforcement learning (RL) algorithms for the data-driven optimal control \cite{sutton2018reinforcement}. In this paper, we present an alternate approach based on the dual formulation of the OCP. This dual approach leads to a convex optimization formulation of the OCP, which can be solved using a single-shot algorithm. This is in contrast to the iterative scheme used for solving the HJB equation.
Furthermore, the iterative algorithm required for solving the HJB equation requires an initial stabilizing controller. Finding stabilizing controller for a nonlinear system is, in general, a nontrivial problem. However, the computational framework for solving the OCP problem in the dual space does not require an initial stabilizing controller. 

The dual formulation to the OCP we present is based on the theory of linear operators, namely the P-F and Koopman operators \cite{Lasota} and is developed in \cite{vaidyaocpconvex}. However, there are differences between the results presented in this paper and \cite{vaidyaocpconvex} as discussed in our contributions. The convex formulation to the OCP in the dual space of densities and occupation measure has been extensively studied in \cite{henrion2013convex,korda2017convergence,lasserre2008nonlinear}. The computational framework in these works relies on moment-based relaxation of the infinite-dimensional optimization problem. In contrast, our proposed computational framework uses data and depends on the linear operator theory for the finite-dimensional approximation of the infinite-dimensional convex optimization problem. The convex formulation for optimal control is also extended to study stochastic OCP, control design with safety constraints, and data-driven stabilization problems \cite{safetyPF,choi2020convex1}. There is a growing body of literature on the use of the Koopman operator for data-driven control, where the control dynamical system is lifted in the space of functions or observables using the Koopman operator \cite{kaiser2021data,huang2018feedback,arbabi2018data,ma2019optimal,korda2020optimal1,mauroy2013spectral,huang2020data}. However, in this paper, we lift the control system using the P-F operator, which is dual to the Koopman operator. Unlike Koopman-based lifting, P-F lifting of the control dynamical system leads to a convex formulation of the OCP \cite{raghunathan2013optimal,vaidya2010nonlinear}. 

The main contributions of this paper are stated as follows. First, we provide a convex formulation to the infinite horizon OCP with discounted cost involving continuous-time dynamics. We consider OCP problems with both positive and negative discount. For the continuous-time OCP, the negative (positive) discount corresponds to the case where the cost function is exponentially decreasing (increasing) with time. There is extensive literature on the OCP with negative discount factor \cite{modares2014linear,modares2016optimal,ghosh1993optimal}. One of the main contributions of this paper is to provide condition for the existence of optimal control problem with a positive discount. The condition arises in the form of a stronger notion of almost everywhere exponential stability \cite{Vaidya_TAC}. Unlike \cite{vaidyaocpconvex}, the computation framework relies on the use of polynomial basis for the approximation of linear Koopman operator using generator Extended Dynamic Mode Decomposition (gEDMD) algorithm \cite{klus2020data}. Hence, we employ sum-of-square (SOS) optimization methods for solving a finite-dimensional optimization problem. The finite-dimensional approximation of the infinite-dimensional optimization problem is written as a semi-definite program (SDP). SOS-based optimization toolbox is then used to solve the SDP in a numerically efficient manner. Existing rigorous results for the convergence analysis of Koopman operator in the limit of data and number of basis functions goes to infinity are leveraged to provide convergence analysis of data-driven optimization problem. Simulation results are presented to verify the efficacy of the developed framework. The results presented in this paper are an extension of our conference paper \cite{moyalan2021sum}. In particular, the data-driven framework for optimal control design and theorems involving OCP with discounted cost function are new to this paper. 

The rest of the paper is structured as follows. In Section \ref{sec:background}, we introduce preliminaries and notations used throughout the paper. The main results of the paper on the convex formulation to OCP are presented in Section \ref{section_main}. The SOS and Koopman-based computation framework for the data-driven approximation of the convex optimization problem is discussed in Section \ref{sec:SOS Control}. Simulation results are presented in Section \ref{sec:examples}. Conclusions are presented in Section \ref{sec:conclusion}.

\section{Preliminaries and Notations}\label{sec:background}
\noindent{\bf Notation:} $\mR^n$ denotes the $n$ dimensional Euclidean space  and $\mR^n_{\geq 0}$ is the positive orthant. Given $\bX\subseteq \mR^n$ and $\bY\subseteq \mR^m$, let ${\cal L}_1(\bX,\bY), {\cal L}_\infty(\bX,\bY)$, and ${\cal C}^k(\bX,\bY)$ denote the space of all real valued integrable functions, essentially bounded functions, and space of $k$ times continuously  differentiable functions mapping from $\bX$ to $\bY$ respectively. If the space $\bY$ is not specified then it is understood that the underlying space is $\mR$. ${\cal B}(\bX)$ denotes the Borel $\sigma$-algebra on $\bX$ and ${\cal M}(\bX)$ is the vector space of real-valued measure on ${\cal B}(\bX)$. $\bs_t(\bx)$ denotes the solution of dynamical system $\dot \bx={\bf f}(\bx)$ starting from initial condition $\bx$.  We will use the notation ${\cal N}_\delta$ to denote the $\delta$ neighborhood of the equilibrium point at the origin for some fixed $\delta>0$ and $\bX_1:=\bX\setminus {\cal N}_\delta$. 
\subsection{Koopman and Perron-Frobenius Operators}
Consider a dynamical system
\begin{align}\dot \bx=\bff(\bx),\;\;\bx\in \bX\subseteq \mathbb{R}^n \label{dyn_sys}\end{align} where the vector field is assumed to be ${\bf f}(\bx)\in {\cal C}^1(\bX,\mR^n)$. There are two different ways of linearly lifting the finite dimensional nonlinear dynamics from state space to infinite dimension space of functions, Koopman and Perron-Frobenius operators.   

\noindent{\bf Koopman Operator:} $\mathbb{K}_t :{\cal L}_\infty(\bX)\to {\cal L}_\infty(\bX)$ for dynamical system~\eqref{dyn_sys} is defined as 
\[[\mathbb{K}_t \varphi](\bx)=\varphi(\bs_t(\bx)),\;\;\varphi\in {\cal L}_\infty. 
\]
The infinitesimal generator for the Koopman operator is
\begin{equation}
\lim_{t\to 0}\frac{\mathbb{K}_t\varphi-\varphi}{t}=\bff(\bx)\cdot \nabla \varphi(\bx)=:{\cal K}_{\bff} \varphi. \label{K_generator}
\end{equation}
\noindent{\bf Perron-Frobenius Operator:} $\mathbb{P}_t:{\cal L}_1(\bX)\to {\cal L}_1(\bX)$ for  system~\eqref{dyn_sys} is defined as
\begin{equation}[\mathbb{P}_t \psi](\bx)=\psi(\bs_{-t}(\bx))\left|\frac{\partial \bs_{-t}(\bx) }{\partial \bx}\right|,\;\;\psi\in {\cal L}_1,\label{pf-operator} 
\end{equation}
where $\left|\cdot \right|$ stands for the determinant. The infinitesimal generator for the P-F operator is given by 
\begin{align}
\lim_{t\to 0}\frac{\mathbb{P}_t\psi-\psi}{t}=-\nabla \cdot (\bff(\bx) \psi(\bx))=: {\cal P}_{\bff}\psi. \label{PF_generator}
\end{align}
These two operators are dual to each other where the duality is expressed as
\begin{align*}
\int_{\mathbb{R}^n}[\mathbb{K}_t \varphi](\bx)\psi(\bx)d\bx=
\int_{\mathbb{R}^n}[\mathbb{P}_t \psi](\bx)\varphi(\bx)d\bx.
\end{align*}

\subsection{Sum of squares} \label{sec:SOS}
Sum of squares (SOS) optimization \cite{Topcu_10,Pablo_03_SDP,parrilo2003minimizing,Pablo_2000} is a relaxation of positive polynomial constraints appearing in polynomial optimization problems. SOS polynomials are in a set of polynomials that can be described as a finite linear combination of monomials, i.e., $p = \sum_{i=1}^\ell d_i p_i^2$ where $p$ is a SOS polynomial, $p_i$ are polynomials, and $d_i$ are nonnegative coefficients. Hence, SOS is a sufficient condition for the nonnegativity of a polynomial. Thus SOS relaxation provides a lower bound on the minimization problems of polynomial optimizations. Using the SOS relaxation, a large class of polynomial optimization problems with positive constraints can be formulated as SOS optimization as
\begin{align} \label{eq:SOSOPT}
\begin{split}
    \min_{\bd} \, \mathbf{w}^\top \mathbf{d} \,\,\, \mathrm{s.t.} \,\,\, p_s(\bx,\bd) \in \Sigma[\bx], \, p_e(\bx;\bd) = 0,
\end{split}
\end{align}
where $\Sigma[\bx]$ denotes SOS set, $\mathbf{w}$ is weighting coefficient, $p_s$ and $p_e$ are polynomials parametrized by coefficients $\bd$. The problem in~\eqref{eq:SOSOPT} can be translated into a Semidefinite Programming (SDP)~\cite{Pablo_03_SDP, Laurent_2009}. There are readily available SOS optimization packages such as SOSTOOLS~\cite{sostools_Parrilo} and SOSOPT~\cite{Seiler_2013_SOSOPT} for solving~\eqref{eq:SOSOPT}.

\subsection{Almost everywhere  uniform stability and Stabilization}

This section derives results on a stronger notion of stability used in formulating optimal control problem with discounted cost. We first present the notion of, a.e., uniform stability as introduced in \cite{rajaram2010stability}. In the rest of the paper, we will use the following notation.

\begin{definition}\label{definitiona.euniform}[a.e. uniform stability]
The equilibrium point of \eqref{dyn_sys} is said to be a.e. uniform  stable w.r.t. measure $\mu\in {\cal M}(\bX)$ if, for every given $\epsilon$, there exists a time $T(\epsilon)$ such that
\begin{align}
\int_{T(\epsilon)}^\infty \mu (B_t)dt<\epsilon\label{eq_aeunifrom}
\end{align}
where $B_t:=\{\bx \in \bX: \bs_t(\bx)\in B\}$ for any set $B\in {\cal B}(\bX_1)$. 
\end{definition}
The above stability definition essentially means that given any arbitrary set $B$ not containing the origin, the measure of the set of all initial conditions that stay inside $B$ can be made arbitrarily small after a sufficiently long time. Note that the above definition of a.e. uniform stability is stronger than the almost everywhere stability notion as introduced in \cite{Rantzer01} (refer to \cite{rajaram2010stability} for the proof). The following definition of a.e. exponential stability is introduced here and is stronger than the above Definition \ref{definitiona.euniform}.  The following exponential stability definition is a continuous-time counterpart of the discrete-time definition studied in \cite{Vaidya_TAC}.

\begin{definition}
\label{definition_aeexponential}[a.e. exponential stability]
The equilibrium point is said to be almost everywhere  exponential stable w.r.t. measure $\mu$
with rate $\gamma>0$ if there exists a constant $M$ such that 
\begin{align}
\mu(B_t)\leq M e^{-\gamma t}
\label{a.eexponential}
\end{align}
where $B_t:=\{\bx\in \mR^n: \bs_t(\bx)\in B\}$ for any set $B\in {\cal B}(\bX_1)$.
\end{definition}



 In the following, we state theorems providing necessary and sufficient condition for a.e. uniform and a.e uniform exponential stability. These results are proved under the following assumption on the equilibrium point of  (\ref{dyn_sys}). 
\begin{assumption}\label{assume_localstability} We assume that $\bx=0$ is locally stable equilibrium point for the system (\ref{dyn_sys}) with local domain of attraction denoted by $\mathcal{D}$ and let  $0\in {\cal N}_\delta\subset \mathcal{D}$.
\end{assumption}
\begin{theorem}\label{theorem_necc_suffuniform}
The equilibrium point $\bx=0$ for the system (\ref{dyn_sys})  satisfying Assumption \ref{assume_localstability} is almost everywhere uniform stable w.r.t. measure $\mu$ if and only if there exists a density function $\rho(\bx)\in{\cal C}^1(\bX\setminus \{0\},\mR^+)$ which is integrable on $\bX_1$ and satisfies 
\begin{align}
\nabla\cdot ({\bf f}\rho )=h_0\label{steady_pde_uniform}
\end{align}
where $h_0\in {\cal C}^1(\bX)$ is the density function corresponding to the measure $\mu$.
\end{theorem}
Refer to \cite[Theorem 13]{rajaram2010stability} for the proof.

\section{Convex Formulation of Optimal Control Problem}\label{section_main}
In this section we briefly summarize the main results from \cite{vaidyaocpconvex} on the convex formulation of the optimal control problem. Consider a control affine system of the form 
\begin{align} \label{cont_syst1}
\dot \bx=\bar {\bf f}(\bx)+{\bf g}(\bx)\bar \bu,\;\;\bx \in \bX\subseteq\mR^n
\end{align}
where, $\bx$ is the state, $\bar \bu=[\bar u_1,\ldots,\bar u_m]^\top\in \mathbb{R}^m$ is the control input and $\bg(\bx)=(\bg_1(\bx),\ldots,\bg_m(\bx))$. All the vector fields are assumed to belong to ${\cal C}^1(\bX,\mR^n)$.

\begin{remark}
The affine control assumption for a dynamical control system is not restrictive as any non-affine dynamical control system can be converted to control-affine by extending the state space. In particular, consider the control dynamical system of the form
\[\dot \bx=\bff(\bx,\bu)\]
then we can define $\bu$ as a new state and introduce $\tilde \bu$ as another control input to write the above system as the following affine in the input control system
\begin{align}
&\dot \bx=\bff(\bx,\bu)\nonumber,\;\;\;\dot \bu=\tilde \bu
\end{align}
\end{remark}
The following assumption is made on (\ref{cont_syst1}).
\begin{assumption}\label{assume_localstable}
We assume that the linearization of the nonlinear system dynamics (\ref{cont_syst1}) at the origin is stabilizable i.e., the pair $(\frac{\partial \bar {\bf f}}{\partial \bx}(0),\bg(0))$ is stabilizable. 
\end{assumption}
Using the above stabilizability assumption, we can design a local stable controller using data. The detailed procedure for the design of such controller is given in Section \ref{sec:SOS Control}. Let $\bu_\ell$ be the locally stable controller. Defining ${\bf f}(\bx):=\bar {\bf f}(\bx)+\bg(\bx) \bu_\ell$ and $\bu=\bar \bu-\bu_\ell$, we can rewrite control system (\ref{cont_syst1}) as
\begin{align}
\dot \bx={\bf f}(\bx)+{\bf g}(\bx) \bu.\label{cont_syst}
\end{align}
The following is valid for the above dynamical system.  With the control input $\bu=0$, the origin of system (\ref{cont_syst}) is  almost sure asymptotically stable locally in small neighborhood $\mathcal{D}$ of the origin such that ${ B}_\delta \subset {\cal D}$.




Consider the discounted cost OCP of the form
\begin{align}
J^\star(\mu_0)\!\!=\!&\inf_\bu\!\!\int_{\bX_1}\!\!\left[\!\int_0^\infty\!\!\! e^{\gamma t}(q(\bx(t))+ \beta \bu(t)^\top \bR \bu(t)) \;dt\right] d\mu_0\nonumber\\
&{\rm subject\;to\;(\ref{cont_syst})}\label{cost_function}
\end{align}
where $\gamma\in \mR$. The existing literature on OCP with discounted cost address the case where $\gamma$ is negative, i.e., negative discount factor. In this paper, with the stronger notion of a.e. uniform stability with geometric decay, we can address the case of cost with a positive discount. Note that the cost function is a function of initial measure $\mu_0$, and this dependency on $\mu_0$ can be explained as follows.
The cost function can be written as 
\begin{align}J(\mu_0)=\int_{\bX_1}V(\bx)d\mu_0(\bx)\label{costmu}
\end{align}
where, $V(\bx)$ can be written as 
\[V(\bx)=\int_0^\infty e^{\gamma t}(q(\bx(t))+\beta \bu^\top (t)\bR \bu(t))dt\]
with $\bx(\cdot)$ being a trajectory with initial condition $\bx(0)=\bx$.While $V(\bx)$ can be recognized with the familiar cost function used in the formulation of OCP in primal domain, the cost function $J(\mu_0)$ is minimized w.r.t. set of initial condition distributed with initial measure $\mu_0$. In the rest of the paper we assume that the initial measure $\mu_0$ is equivalent to Lebesgue with density function $0<h_0(\bx)\in  {\cal L}_1(\mR^n,\mR_{> 0})\cap {\cal C}^1(\mR^n)$. 
We make the following assumption on the OCP. 
\begin{assumption}\label{assumption_onocp}
 We assume that the state cost function $q: \mR^n\to \mR_{\geq 0}$ is zero at the origin and uniformly bounded away from zero outside the neighborhood ${\cal N}_\delta$ and $\bR>0$. Furthermore, there exists a feedback control for which the cost function in (\ref{cost_function}) is finite and  that the optimal control is feedback in nature, i.e., $\bu^\star=\bk^\star(\bx)$ with the function $\bk^\star$
 being in ${\cal C}^1(\bX,\mR^m)$. 
\end{assumption}
With the assumed feedback form of the optimal control input, the OCP can be written as 
{\small
\begin{eqnarray}
 \inf\limits_{\bk\in {\cal C}^1(\bX)} &\int_{\bX_1}\left[\int_0^\infty e^{\gamma t}(q(\bx(t))+ \beta \bk(\bx(t))^\top \bR \bk(\bx(t)))\;dt\right] d\mu_0\nonumber\\
 {\rm s.t.}&\dot \bx={\bf f}(\bx)+\bg(\bx)\bk(\bx)
\label{ocp_main_discounted}
\end{eqnarray}
}
The following is the main theorem on the OCP with discounted cost function.

\begin{theorem}\label{theorem_maingeometric}
Consider the OCP (\ref{ocp_main_discounted}) with discount factor $\gamma\leq 0$ and assume that the cost function and optimal control satisfy Assumption \ref{assumption_onocp}. Then the OCP (\ref{ocp_main_discounted}) can be written as the following infinite dimensional convex optimization problem 
\begin{eqnarray}
J^\star(\mu_0)&=&\inf_{\rho\in {\cal S},\bar \brho\in {\cal C}^1(\bX_1)} \;\;\; \int_{\bX_1} q(\bx)\rho(\bx)+\beta\frac{\bar \brho(\bx)^\top \bR\bar \brho(\bx)}{\rho} d\bx\nonumber\\
{\rm s.t}.&&\nabla\cdot ({\bf f}\rho +\bg\bar \brho)=\gamma \rho+ h_0
\label{eqn_ocpdiscount_L2}
\end{eqnarray}
where $\bar \brho=(\bar \rho_1,\ldots, \bar\rho_m)$ and ${\cal S}:={\cal L}_1(\bX_1)\cap {\cal C}^1(\bX_1,\mR_{\geq 0})$. The optimal feedback control input is recovered from the solution of the above optimization problem as  
\begin{align}
\bk^\star(\bx)=\frac{\bar \brho^\star(\bx)}{\rho^\star(\bx)}\label{feedback_input}.
\end{align}
Furthermore, if $\gamma=0$, then optimal control $\bk^\star(\bx)$ is a.e. uniformly stabilizing   w.r.t. measure $\mu_0$.
\end{theorem}

Proof of theorem \eqref{theorem_maingeometric} is given in Appendix.

Next, we consider discounted cost OCP with ${\cal L}_1$ norm on control term 
\begin{eqnarray}
 \inf\limits_{\bk\in {\cal C}^1(\bX)} &\int_{\bX_1}\left[\int_0^\infty e^{\gamma t}(q(\bx(t))+ \beta \|\bk(\bx(t))\|_1)\;dt\right] d\mu_0(\bx)\nonumber\\
 {\rm s.t.}&\dot \bx={\bf f}(\bx)+\bg(\bx)\bk(\bx).
\label{ocp_main_discounted1}
\end{eqnarray}
We make the following assumption on the nature of optimal control for the ${\cal L}_1$-norm OCP (\ref{ocp_main_discounted1}). 
\begin{assumption}\label{assume_OCP1}
We assume that the state cost function $q: \mR^n\to \mR_{\geq 0}$ is zero at the origin and uniformly bounded away from zero outside the neighborhood ${\cal N}_\delta$ and $\bR>0$. Furthermore, there exists a feedback control input for which the cost function in \eqref{ocp_main_discounted1} is finite. Furthermore,  the optimal control is feedback in nature, i.e., $\bu^\star=\bk^\star(\bx)$ with the function $\bk^\star$ is assumed to be ${\cal C}^1(\bX,\mR^m)$. 
\end{assumption}
\begin{theorem}\label{theorem_maingeometric1}
Consider the OCP (\ref{ocp_main_discounted1}) with discount factor $\gamma\leq 0$ and assume that the cost function and optimal control satisfy Assumption \ref{assume_OCP1} respectively. Then the OCP (\ref{ocp_main_discounted1}) can be written as following infinite dimensional convex optimization problem 
\begin{eqnarray}
J^\star(\mu_0)&=&\inf_{\rho\in {\cal S},\bar \brho\in {\cal C}^1(\bX_1)} \;\;\; \int_{\bX_1} q(\bx)\rho(\bx)+\beta\|\bar \brho(\bx)\|_1 d\bx\nonumber\\
{\rm s.t}.&&\nabla\cdot ({\bf f}\rho +\bg\bar \brho)=\gamma \rho+ h_0
\label{eqn_ocpdiscount}
\end{eqnarray}
where $\bar \brho=(\bar\rho_1,\ldots, \bar\rho_m)$ and ${\cal S}:={\cal L}_1(\bX_1)\cap {\cal C}^1(\bX_1,\mR_{\geq 0})$. The optimal feedback control input is recovered from the solution of the above optimization problem as  
\begin{align}
\bk^\star(\bx)=\frac{\bar \brho^\star(\bx)}{\rho^\star(\bx)}.
\end{align}
Furthermore, if $\gamma=0$, then optimal control $\bk(\bx)$ is a.e. uniformly stabilizing w.r.t. measure $\mu_0$.
\end{theorem}

\begin{proof}
The proof of Theorem \ref{theorem_maingeometric1} follows along similar lines to the proof of Theorem \ref{theorem_maingeometric}. 
\end{proof}

\begin{remark}\label{remark_singularity}
It is important to emphasize that the optimal feedback controller with discount factor $\gamma= 0$ is stabilizing in almost everywhere sense. This is analogous to the optimal control design in the primal formulation. The optimal cost function also serves as a Lyapunov function, thereby ensuring the stability of the feedback control system. In our proposed dual setting, the optimal density function serves as a.e. stability certificate for the case of discount factor $\gamma= 0$. However, due to the dual nature of the Lyapunov function and density function \cite{Vaidya_TAC,Rantzer01}, the density function has a singularity at the origin. Because of this singularity at the origin, the cost function is evaluated in $\bX_1$ excluding the small region around the origin. Hence it may become necessary to design a local stabilizing or local optimal controller. The existence of such a local stabilizing controller is ensured following Assumption \ref{assume_localstable}. The local controller, say $\bk_\ell$, can be blended with global control $\bk^\star$ using the following  formula  \cite{rantzer2001smooth}. 

\begin{equation}
    u(\bx) = \frac{\rho_L}{\rho_L+\rho_N}\bk_\ell(\bx) + \frac{\rho_N}{\rho_L+\rho_N}\bk^*(\bx) \nonumber
\end{equation}
 
\begin{equation}
    \rho_L(\bx) = max\{((\bx^T P \bx))^{-3}-\Delta,0\} \nonumber
\end{equation}
where matrix $P > 0$ define a control Lyapunov function. The parameter $\Delta$ determines the region of operation for the local controller. 
\end{remark}



The optimal control results involving ${\cal L}_2$ and ${\cal L}_1$ norm on the control input with positive discount factor, i.e., $\gamma>0$, are proved under the following assumption. 

\begin{assumption}\label{assumption_ocppositivediscount}
We assume that the state cost function $q: \mR^n\to \mR_{\geq 0}$ is zero at the origin and uniformly bounded away from zero outside the neighborhood ${\cal N}_\delta$ and $\bR>0$. Furthermore, there exists a feedback control for which the cost function in (\ref{cost_function}) is finite and  that the optimal control is feedback in nature, i.e., $\bu^\star=\bk^\star(\bx)$ with the function $\bk^\star$
 being in ${\cal C}^1(\bX,\mR^m)$. Furthermore, the feedback controller is assumed to be almost everywhere exponentially stabilizing  (Definition \ref{definition_aeexponential}) with decay rate $\gamma'>\gamma>0$.
\end{assumption}
Note that Assumption \ref{assumption_ocppositivediscount} is same as Assumption \ref{assumption_onocp} except for the additional requirment that the feedback controller is a.e. exponentially stabilizing with decay rate strictly large than $\gamma$.

\begin{theorem}\label{theorem_maingeometric_positivediscount}
Consider the OCP (\ref{ocp_main_discounted}) with discount factor $\gamma> 0$ and assume that the cost function and optimal control satisfy Assumption \ref{assumption_ocppositivediscount}. Then the OCP (\ref{ocp_main_discounted}) can be written as the following infinite dimensional convex optimization problem 
\begin{eqnarray}
J^\star(\mu_0)&=&\inf_{\rho\in {\cal S},\bar \brho\in {\cal C}^1(\bX_1)}\int_{\bX_1} q(\bx)\rho(\bx)+\beta\frac{\bar \brho(\bx)^\top \bR\bar \brho(\bx)}{\rho} d\bx\nonumber\\
{\rm s.t}.&&\nabla\cdot ({\bf f}\rho +\bg\bar \brho)=\gamma \rho+ h_0
\label{eqn_ocpdiscount_L2}
\end{eqnarray}
where $\bar \brho=(\bar \rho_1,\ldots, \bar\rho_m)$ and ${\cal S}:={\cal L}_1(\bX_1)\cap {\cal C}^1(\bX_1,\mR_{\geq 0})$. The optimal feedback control input is recovered from the solution of the above optimization problem as  
\begin{align}
\bk^\star(\bx)=\frac{\bar \brho^\star(\bx)}{\rho^\star(\bx)}\label{feedback_input}.
\end{align}
\end{theorem}
The proof of this theorem is provided in the Appendix. Theorem analogous to Theorem \ref{theorem_maingeometric1} can be stated and proved for the case involving ${\cal L}_1$ control norm and with positive discount factor $\gamma>0$. \\

\begin{remark}
In the above formulations of the OCPs, we did not explicitly impose constraints on the control input. Explicit constraints on the magnitude of the control input can be imposed in a convex manner as follows:
\begin{align}
   \|\bu\|_1\leq M\iff   |\bar \rho_1(\bx)|^2+\ldots+|\bar \rho_m(\bx)|^2\leq M \rho(\bx),\label{control_constraints}
\end{align}
for some positive constant $M$. To arrive at (\ref{control_constraints}) we have used the formula for the optimal feedback control, i.e., (\ref{feedback_input}) and the fact that $\rho(\bx)>0$.  
The above constraints are linear in the optimization variables $\bar \brho$ and $\rho$ and hence can be implemented convexily. So the OCP involving explicit norm constraints on the control input can be implemented convexily by augmenting the optimization problem  with linear constraints in (\ref{control_constraints}). 
\end{remark}

\begin{remark}
In the above formulations of the OCP problem, we have assumed that density function $h_0>0$, implying that the initial measure $\mu_0$ is equivalent to Lebesgue. This assumption is necessary as $h_0>0$ guarantees that $\rho>0$ (refer to Eq.  (\ref{definingrho})) and hence the  the feedback control input $\bk=\frac{\bar \brho}{\rho}$ is well defined. However, it is possible to relax this assumption and work with density function $h_0\geq 0$. This will correspond to the case where the initial measure $\mu_0$ is continuous w.r.t. Lebesgue measure and not equivalent to Lebesgue. In order to ensure that the feedback control input is well defined when $\mu_0$ is continuous w.r.t. Lebesgue measure, we need to impose the following constraints on the control input
\[  |\bar \rho_k(\bx)|^2\leq M \rho(\bx),\;\;k=1,\ldots,m\]
for some large constant $M$. The above constraints will ensure that the for a.e. $\bx$ if $\rho(\bx)=0\implies \bar \brho_k(\bx)=0$ for $k=1,\ldots, m$ thereby the feedback control input is well defined. Working with absolutely continuous initial measure $\mu_0$ or equivalently $h_0\geq 0$ will correspond to the case where optimality is guaranteed only from set of initial condition with support on $\mu_0$.      
\end{remark}
In the following, we demonstrate how the results involving dual formulation to the OCP problem works out for the special case of scalar linear system. The main conclusion is that the optimal control obtained using dual formulation matches with the control obtained using a primal formulation of OCP, namely the linear quadratic regulator problem for a particular choice of $h_0(\bx)$. Note that the initial measure or the density function $h_0(\bx)$ is unique to our dual formulation with no parallel in the primal formulation. 

\section{Koopman and SOS-based Computation Framework for Optimal Control} \label{sec:SOS Control}
This section provides Koopman and SOS-based computational framework for the finite-dimensional approximation of OCP involving ${\cal L}_1$/${\cal L}_2$ costs. We begin with the following parameterization for the optimization variables $\rho$ and $\bar\brho$. 
\begin{equation}
    \rho(\bx) = {a(\bx)}/b(\bx)^{\alpha},\;\; \bar{\brho}(\bx) = \bc(\bx)/b(\bx)^{\alpha}, \label{rational_parametrization}
\end{equation}
where $a(\bx)\geq 0$  and $\bc(\bx) = \begin{bmatrix} c_1(\bx),\, \ldots,\, c_m(\bx)\end{bmatrix}^\top$. Here, $b(\bx)$ is a positive polynomial (positive at $\bx \ne 0$), and $\alpha$ is a positive constant which is sufficiently large for integrability condition. In fact $b(\bx)$ is chosen to be control Lyapunov function based on the  linearized control dynamics at the origin. The data-driven procedure for the identification of the linear dynamics used in the construction of $b(\bx)$ is explained in Remark \ref{remark_bconstrucntion}. The particular form for the parameterization of the optimization variable in (\ref{rational_parametrization}) is chosen because of the fact that $\rho$ has singularity at the origin (Remark \ref{remark_singularity}). Using~\eqref{rational_parametrization}, we can write the constraints for the optimization problem as  in~\eqref{ocp_main_discounted1} and~\eqref{eqn_ocpdiscount} as
\begin{align*} 
h &=\nabla  \cdot (\bff\rho  + \bg\bar{\brho}) - \gamma \rho = \nabla  \cdot \left[(\bff a + \bg\bc)/b^{\alpha} \right] - \gamma \frac{a}{b^{\alpha}}\\
&=\frac{1}{b^{\alpha +1}} [(1 + \alpha) b \nabla  \cdot (\bff a + \bg\bc) - \alpha \nabla \cdot (b\bff a + b\bg\bc)\\
&\;\;\;\;-\gamma a b].
\end{align*}
With the above form of the constraints, we assume the following parameterization for $h=\frac{d}{b^{\alpha+1}}$, 
where $d$ is an arbitrary positive polynomial. With the assumed form of $h$, we write the constraints in the optimization variable, $a$ and $\bc$ 
 as
\begin{align} \label{rational_parametrization_1}
(1 + \alpha) b \nabla  \cdot (\bff a + \bg\bc) - \alpha \nabla \cdot (b\bff a + b\bg\bc) - \gamma a b = d
\end{align}
The above constraint in the optimization problem can be written in terms of the P-F generator as follows:
\begin{align}
    &(1+\alpha)b \left({\cal P}_{\bf f}a+\sum_{i=1}^m{\cal P}_{\bf g_i}c_i\right)-\alpha \left({\cal P}_{\bf f}(ba)+\sum_{i=1}^m{\cal P}_{\bf g_i}(bc_i)\right)\nonumber\\
    &-\gamma a=d\label{ss12}
\end{align}
\subsection{Data-driven Approximation of the Generators}\label{section_datagenerator}
From (\ref{ss12}), it follows that the data-driven approximation of the constraints in the optimization problems (\ref{eqn_ocpdiscount}) and (\ref{eqn_ocpdiscount_L2}) involves approximation of the P-F generators, ${\cal P}_{\bf f}$ and ${\cal P}_{\bg_i}$. Furthermore, the P-F generator can be expressed in terms of the Koopman generator as 
\begin{eqnarray}
-{\cal P}_\bff \psi \hspace{-0.02in} = \hspace{-0.02in} \nabla\cdot(\bff \psi) \hspace{-0.02in} = \hspace{-0.02in} \bff \cdot \nabla \psi+\nabla\cdot \bff \psi \hspace{-0.02in} = \hspace{-0.02in} {\cal K}_\bff \psi+\nabla \cdot \bff \psi.\label{PF_gen_approx}
\end{eqnarray}

Expressing the P-F generator in terms of the Koopman generator allows us to use data-driven methods used to approximate the Koopman generator for the approximation of the P-F generator. In particular, we use generator Extended Dynamic Mode Decomposition (gEDMD) algorithm from~\cite{klus2020data} for the approximation.
To approximate Koopman generators, we first collect time-series data from the dynamical system in~\eqref{cont_syst1} by injecting different control inputs: i) zero control inputs, $\bu=0$, and ii) unit step control inputs, $\bu=\mathbf{e}_j$\footnote{$\be_j \in \mathbb{R}^m$ denotes unit vectors, i.e., $j$th entry of $\mathbf{e}_j$ is 1, otherwise 0.} for $j=1,\ldots,m$ for a finite time horizon with sampling step $\delta t$. Let, 
\begin{align} \label{eq:data_matrix}
    \mathbf{X}_i = \left [ \mathbf{x}_1, \ldots, \mathbf{x}_{T_i} \right ], \,\, \dot{\mathbf{X}}_i = \left [ \dot{\mathbf{x}}_1, \ldots, \dot{\mathbf{x}}_{T_i} \right ],
\end{align}
with $i=0,1,\ldots,m$ for zero and step control inputs where $T_i$ are the number of data points for $i$th input case. The samples in $\mathbf{X}_i$ do not have to be from a single trajectory; it can be a concatenation of multiple experiment/simulation trajectories. Also, time derivatives of the states $\dot{\bx}$ can be accurately estimated using numerical algorithms such as finite differences. Next, we construct a polynomial basis vector:
\begin{align} \label{eq:basis functions}
\bPsi(\bx) = [\psi_1(\bx), \ldots, \psi_Q(\bx)]^\top,
\end{align}
which can include monomials or Legendre/Hermite polynomials. The data-driven approximation of the generator will essentially involve the projection of the infinite-dimensional Koopman and P-F generators on the finite-dimensional space spanned by the basis function (\ref{eq:basis functions}). Following \cite{klus2020data}, we define:
\begin{align} \label{eq:basis functions time derivatives}
\dot{\bPsi}(\bx,\dot{\bx}) = [\dot{\psi}_1(\bx,\dot{\bx}), \ldots, \dot{\psi}_Q(\bx,\dot{\bx})]^\top,
\end{align}
\begin{align}
    \dot{\psi}_k(\bx,\dot{\bx}) = (\nabla_\bx \psi_k)^\top \dot{\bx} = \textstyle{\sum}_{j=1}^n \frac{\partial \psi_k}{\partial x_j} \frac{dx_j}{dt}
\end{align}
The partial derivatives of the basis function are computed analytically which is required for $\dot{\psi}_k(\bx,\dot{\bx})$. Note that we also need $\frac{dx_j}{dt}$ which is simply denoted by $\dot{x}_j$. The value of $\dot{x}_j$ is  approximated using finite differences:
\begin{align}
    \dot{x}_j \approx \frac{x_j - x_{j-1}}{\Delta t}
\end{align}
where $x_{j-1}$ and $x_j$ are the $j-1^{th}$ and $j^{th}$ data point in the system trajectory and $\Delta t$ is the time difference between two consecutive data points. A more sophisticated finite-difference method can be used, e.g., total variation regularization~\cite{Chartrand2011-ni} for noisy data and discontinuous derivatives, and also second-order central difference for better accuracy. Then, the Koopman generator approximate $\mathbf{L}_i$ for each input case can be approximated as:
\begin{align} \label{eq:finite Koopman generator}
\begin{split}
    \mathbf{L}_i &= \underset{\bL_i}{\mathrm{argmin}} \, ||\bB_i - \bA_i \bL_i||_F,\\
      \mathbf{A}_i &= \frac{1}{T_i} \textstyle{\sum}_{\ell=1}^{T_i} \, \boldsymbol{\Psi}(\bX_{i,\ell}) \boldsymbol{\Psi}(\bX_{i,\ell})^\top,\\
        \mathbf{B}_i &= \frac{1}{T_i} \textstyle{\sum}_{\ell=1}^{T_i} \, \boldsymbol{\Psi}(\bX_{i,\ell}) \dot{\boldsymbol{\Psi}}(\bX_{i,\ell},\dot{\bX}_{i,\ell})^\top,
\end{split}
\end{align}
and $\bX_{i,\ell}$ and $\dot{\bX}_{i,\ell}$ denote $\ell$th column of $\bX_i$ and $\dot{\bX}$, respectively.  The solution of~\eqref{eq:finite Koopman generator} is explicitly known, $\mathbf{K}_i = \mathbf{A}_i^\dagger \mathbf{B}_i$, where $\dagger$ stands for pseudo-inverse. Given the Koopman generator approximates for $\bff$, ${\cal K}_{\bff}\approx \bL_0$, using the linearity of the generators,
\begin{equation}\label{eq:L0}
{\cal K}_{\bg_j}\approx {\bL}_j-\bL_0,\;\;\;j=1,\ldots, m.
\end{equation}

The above is one method to estimate ${\cal K}_{\bff}$ and ${\cal K}_{\bg_j}$. They can also be approximated jointly by using trajectories subject to arbitrary inputs.
Next, we approximate the divergence of vector field $\bff$ as
\begin{equation}
\nabla \cdot {\bff} =\nabla \cdot [{\cal K}_0 x_1,\ldots, {\cal K}_0 x_n]^\top \approx \nabla \cdot(\bC_x^\top \bL_0 \bPsi) \label{approx_divergence}
\end{equation}
where $\bC_x$ is a coefficient vector for $\bx$, i.e., $\bx = \bC_x^\top \bPsi$, which can be found easily if $\bPsi$ includes 1st-order monomials (i.e., $\bx$). 
Similarly, the divergence of vector fields $\bg_j$ are approximated as 
\begin{equation}
\nabla\cdot (\bg_j)\approx \nabla \cdot(\bC_x^\top \bL_j \bPsi),\;\;j=1,\ldots,m.\label{divg_approx}
\end{equation}
from~\eqref{PF_gen_approx},~\eqref{eq:L0}--\eqref{divg_approx}, P-F generators are approximated by
\begin{equation} \label{eq:PF generators}
{\bP}_j\hspace{-0.03in}=\hspace{-0.03in}\bL_j+\nabla \cdot(\bC_x^\top \bL_j \bPsi) {\bf I}
\end{equation}
for $j=0,1,\ldots,m$. 
\begin{remark}
While the above procedure describes an approach for the approximation of the Koopman generators corresponding to the drift vector field, ${\bf f}(\bx)$, and control vector fields, $\bg_j(\bx)$, for $j=1,\ldots,m$ using zero input and step input, it is also possible to identify these vector field using random inputs. The problem of data-driven identification of the system dynamics using random or arbitrary control input will involve identifying bilinear vector fields. It can again be reduced to a least-square optimization problem similar to (\ref{eq:finite Koopman generator}).
\end{remark}

To parameterize the optimization variables, we express polynomial functions $a(\bx)$, $b(\bx)$, and $c_j(\bx)$ with respect to the basis $\bPsi(\bx)$. Let $\bC_a$, $\bC_b$, and $\bC_{c_j}$ be the coefficient vectors used in the expansion of $a(\bx)$, $b(\bx)$, and $c_j(\bx)$:
\begin{align} \label{eq:a&cj}
    a(\bx) \hspace{-0.03in}=\hspace{-0.03in} \bC_a^\top \bPsi(\bx), b(\bx) \hspace{-0.03in}=\hspace{-0.03in} \bC_b^\top \bPsi(\bx), c_j(\bx)\hspace{-0.03in}=\hspace{-0.03in}\bC_{c_j}^\top \bPsi(\bx). 
\end{align}
Note that $\bC_b$ contains constant coefficients of $b(\bx)$ since $b(\bx)$ is a known polynomial function. Similarly, let $\bC_{ab}$, $\bC_{bc_1}$, \ldots, $\bC_{bc_m}$ denote coefficient vectors of polynomials, $a(\bx) b(\bx)$, and $b(\bx) c_j(\bx)$, for $j=1,\ldots,m$, namely,
\begin{align} \label{eq:ab&bc}
    a(\bx)b(\bx) = \bC_{ab}^\top \bPsi (\bx), b(\bx) c_j(\bx) = \bC_{bc_j}^\top \bPsi (\bx).
\end{align}
Constructing the coefficient vectors $\bC_{ab}$, and $\bC_{bc_j}$ from the coefficient vectors $\bC_{a}$, $\bC_b$, and $\bC_{c_j}$ requires trivial numerical procedures. In case that $\bPsi(\bx)$ is a monomial, finding these coefficient vectors are straightforward. If $\bPsi(\bx)$ is not a monomial vector, it involves more complicated numerical steps. One approach is to express the basis with respect to a common monomial vector $\mathcal{M}(\bx)$ such that $\bPsi(\bx) = \bC_\Psi^\top \mathcal{M}(\bx)$ where $\bC_\Psi$ is a coefficient matrix that connects $\bPsi(\bx)$ and $\mathcal{M}(\bx)$. Then, we can find coefficient vectors in terms of the monomial vector $\mathcal{M}(\bx)$, i.e., $a(\bx) b(\bx) = \bC_{ab}^{'} \mathcal{M}(\bx)$ and $b(\bx) c_j(\bx) = \bC_{bc_j}^{'} \mathcal{M}(\bx)$, and convert these coefficient vectors back to the original basis $\bPsi(\bx)$ by multiplying pseudo-inverse of $\bC_\Psi$, i.e., $\bC_{ab} = \bC_\Psi^\dagger \bC_{ab}^{'}$ and $\bC_{bc_j} = \bC_\Psi^\dagger \bC_{bc_j}^{'}$. The implementation of this procedure can be done easily using polynomial toolbox provided in SOSOPT in Matlab~\cite{Seiler_2013_SOSOPT}.



Now, using approximated infinitesimal PF generators in~\eqref{eq:PF generators}, we restate the LHS of~\eqref{rational_parametrization_1} in~\eqref{eqn_ocpdiscount_L2} and~\eqref{eqn_ocpdiscount} as:
\begin{align} \label{eq:stability_approximation}
\begin{split}
&(1+\alpha) b(\bx) \left ( \bC_a^\top \bP_0 \bPsi(\bx) + \textstyle{\sum}_{j=1}^m \bC_c^\top \bP_j \bPsi(\bx) \right )\\
&- \alpha \left( \bC_{ab}^\top \bP_0 \bPsi(\bx) + \textstyle{\sum}_{j=1}^m \bC_{bc_j}^\top \bP_j \bPsi(\bx) \right) - \gamma a(\bx) b(\bx).
\end{split}
\end{align}

\begin{remark}
In \cite{korda2018convergence,klus2020data}, convergence results for the finite-dimensional approximation of the Koopman operator and generators in the limit as the number of basis functions and data points goes to infinity is studied. These convergence results combined with the finite-dimensional approximation of the cost function for the optimization problem can be used to provide theoretical justification for the solution obtained using the finite-dimensional approximation of the infinite-dimensional optimization problem. 
\end{remark}
In the following sections, we discuss the finite dimensional approximation of the cost function leading up to the finite dimensional approximation of the infinite dimensional optimization problems (\ref{eqn_ocpdiscount}) and (\ref{eqn_ocpdiscount_L2}) involving ${\cal L}_1$ and ${\cal L}_2$ norms on control input respectively.

\subsection{Optimal Control with $\mathcal{L}_1$ norm of feedback control}

Using the assumed paramaterization for the 
$\rho$ and $\bar\brho$ from (\ref{rational_parametrization}) the cost function for the ${\cal L}_1$ OCP problem in~\eqref{eqn_ocpdiscount} can be written as 
\begin{align}\label{eqn_ocp3}
\begin{split}
\inf_{a\geq 0,\bc} &\;\;\; \int_{\bX_1} \frac{q(\bx) a(\bx)}{b(\bx)^\alpha}+\frac{\beta||\bc(\bx)||_1}{b(\bx)^\alpha} d\bx
\end{split}
\end{align}
where a small neighborhood of the origin, ${\cal N} = \{\bx \in \bX: |\bx| \le \epsilon,\, \epsilon > 0\}$, 
is chosen as a polytope and excluded from the integration of the cost function to remove singularity of the density at the origin (refer to Remark \ref{remark_singularity}).


To make~\eqref{eqn_ocp3} solvable, we introduce dummy polynomials $\bs(\bx) = [s_1(\bx),\ldots,s_m(\bx)]^\top$,  adding constraints:
\begin{align} \label{eq:const on sx}
    \bs(\bx) - \bc(\bx) \geq 0,\, \bs(\bx) + \bc(\bx) \geq 0.
\end{align}
The polynomial $\bs$ is expressed in terms of the basis function using the coefficient vector $\bC_{s_j}$ as:
\begin{align} \label{eq:a&cj}
    s_j(\bx)\hspace{-0.03in}=\hspace{-0.03in}\bC_{s_j}^\top \bPsi(\bx),
\end{align}
for $j=1,\ldots,m$.
Substituting (\ref{eq:ab&bc}) in the integral cost (\ref{eqn_ocp3}), we obtain following finite dimensional approximation of the cost function.
\begin{align} \label{eq:L1 cost precomp}
    \bd_1^\top \bC_a + \beta \textstyle{\sum}_{j=1}^{m} \bd_2^\top \bC_{s_j}
\end{align}
where $\bd_1$ and $\bd_2$ are the coefficient vectors given by
\begin{align} \label{eq:integral}
        \bd_1 = \int_{\bX_1} \frac{q(\bx)\bPsi(\bx)}{b(\bx)^\alpha} d\bx,\,\, \bd_2 = \int_{\bX_1} \frac{\bPsi(\bx)}{b(\bx)^\alpha} d\bx.
\end{align}
Using~\eqref{eq:const on sx}--\eqref{eq:integral} and SOS positivity constraints denoted by $\Sigma[\bx]$, \eqref{eqn_ocp3} can be expressed as a SOS problem as:
\begin{align} \label{eqn_ocp5}
\begin{split}
        \min_{\underset{j=1,\ldots,m}{\bC_a, \bC_{c_j}, \bC_{s_j}}} &\,\,\, \bd_1^\top \bC_a + \beta \textstyle{\sum}_{j=1}^{m} \bd_2^\top \bC_{s_j}\\
        &{\rm s.t}. \,\; \eqref{eq:stability_approximation} \in \Sigma[\bx],\,\, a(\bx) \in \Sigma[\bx],\\
        &\, (\bs(\bx) - \bc(\bx)) \in \Sigma[\bx],\, (\bs(\bx) + \bc(\bx)) \in \Sigma[\bx].
\end{split}
\end{align}

\subsection{Optimal Control with $\mathcal{L}_2$ norm of feedback control}
$\mathcal{L}_2$ OCP in~\eqref{eqn_ocpdiscount_L2} is restated as:
\vspace{-0.0in}
\begin{align}\label{eq:eqn_ocp6}
\begin{split}
    \min_{a,\bc} &\quad \int_{\bX_1} \frac{q(\bx)a(\bx)}{b(\bx)^\alpha} + \beta \frac{\bc(\bx)^\top \bR \bc(\bx)}{a(\bx)b(\bx)^\alpha} dx\\
    &\mathrm{s.t.} \quad \eqref{eq:stability_approximation} \geq 0,\,\, a(\bx) \geq 0.
\end{split}
\end{align}
by following the same parameterization in~\eqref{rational_parametrization}. Subsequently, we reformulate \eqref{eq:eqn_ocp6} as follows:
\vspace{-0.0in}
\begin{align}\label{eq:eqn_ocp7}
\begin{split}
    \min_{a,\bc,w} &\quad \int_{\bX_1} \frac{q(\bx)a(\bx)}{b(\bx)^\alpha} + \beta \frac{w(\bx)}{b(\bx)^\alpha} d\bx\\
    \mathrm{s.t.} &\quad \eqref{eq:stability_approximation} \geq 0,\,\, a(\bx) \geq 0,\\
    &\quad \bM(\bx) = \begin{bmatrix} w(\bx) & \bc(\bx)^\top \\ \bc(\bx) & a(\bx)\bR^{-1} \end{bmatrix} \succcurlyeq 0,
\end{split}
\end{align}
where the positive semidefinite (PSD) of $\bM(\bx)$ is a result of applying the Schur complement lemma on $\mathcal{L}_2$ cost bounded by $w(\bx)$, i.e., $\frac{\bc(\bx)^\top \bR \bc(\bx)}{a(\bx)} \leq w(\bx)$. Now, to algebraically express $\bM(\bx)\succcurlyeq 0$, we introduce the lemma:
\begin{lemma}[\scriptsize Positive semidefinite polynomial matrix] ~\cite{Scherer_2006} \label{lem:polyPSD}
A $p \times p$ matrix $\bH(\bx)$ whose entries are polynomials is positive semidefinite with respect to the monomial vector $\bz(\bx)$, \textit{if and only if}, there exist $\bD \succcurlyeq 0$ such that
\begin{align*}
    \bH(\bx) = \left(\bz(\bx) \otimes \bI_p \right)^\top \bD \left(\bz(\bx) \otimes \bI_p\right),
\end{align*}
where $\otimes$ denotes a Kronecker product (tensor product) and $\bI_p$ is an identity matrix with dimension $p$.
\end{lemma}

Following Lemma~\ref{lem:polyPSD}, let $\bz(\bx)$ be a monomial vector with the maximum degree equal to $\mathrm{floor}(\mathrm{deg}(\bPsi(\bx))/2)+1$, then $\bM(\bx)$ in~\eqref{eq:eqn_ocp7} is PSD when there exists $\bD \succcurlyeq 0$ such that $\bM(\bx)=\bH(\bx)$. Using this result and~\eqref{eq:integral}, a SOS problem equivalent to~\eqref{eq:eqn_ocp7} can be formulated as follows:
\begin{align} \label{ocp_9}
        \min_{\underset{j=1,\ldots,m}{\bC_a, \bC_w, \bC_{c_j}}} &\,\,\, \bd_1^\top \bC_a + \beta \bd_2^\top \bC_{w}\nonumber\\
        \mathrm{s.t.} &\quad \eqref{eq:stability_approximation} \in \Sigma[\bx],\,\, a(\bx) \in \Sigma[\bx],\\
        &\quad w(\bx) - \bH_{11}(\bx) = 0,\, \bc(\bx) - \bH_{12}(\bx) = 0,\nonumber\\
        &\quad a(\bx) \bR^{-1} - \bH_{22}(\bx) = 0,\, \bD \succcurlyeq 0\nonumber
\end{align}
where $\bH_{ij}(\bx)$ denotes the ${ij}$th entry of~$\bH(\bx)$; and $\boldsymbol{\mathcal{C}}_w$ is a coefficient vector of~$w(\bx)$, i.e., $w(\bx)=\boldsymbol{\mathcal{C}}_w^\top \bPsi(\bx)$.

\begin{remark}\label{remark_bconstrucntion}
To obtain $b(\bx)$, the control Lyapunov function for the linearized dynamics, we first identify the linearized control system dynamics from time-series data collected near the origin. to identify the linear dynamics, we use the gEDMD algorithm discussed  in Section \ref{section_datagenerator} for the special case of identity basis functions i.e., $\bPsi(\bx)=\bx$. Once linearized system dynamics is identified we use linear quadratic regulator based optimal control for the construction of $b(\bx)$, namely $b(\bx)=\bx^\top P\bx$, where $P$ is the solution of algebraic Riccati equation (ARE). Following Assumption \ref{assume_localstable}, we know that there exists a positive definite solution, $P$, to the ARE, which serves as control Lyapunov function for the linearized control system. 
\end{remark}

\section{Simulation Results}\label{sec:examples}
In this section, we present simulation results to illustrate the proposed data-driven control framework. All the simulation results are performed using MATLAB on a desktop computer with 64GB RAM. We have taken the value of $\alpha = 4 $ and $\beta = 1$ for all our examples. The cost function and control matrix for each example is $q(\bx) = \bx^T\bx$ and $R=1$ respectively. Furthermore, the cost function is computed outside the region $\cal N$ of the neighborhood. However, we did not implement a local stabilizing controller around the origin and hence no blending controller (Remark \ref{remark_singularity}). Also, the maximum degree of $a(\bx)$  is taken to be 1. We take the simulation time step $\Delta t = 0.01 \, \mathrm{[s]}$ for sampling time-series data, and also, Legendre polynomials are used as dictionary functions for all examples. We use SOSOPT~\cite{Seiler_2013_SOSOPT} toolbox to solve the formulated SOS optimization problems for the OCPs in~\eqref{eqn_ocp5} and~\eqref{ocp_9}. All other parameters used for each example are listed in Table \ref{table:OCP param}.

\begin{table}[b]
\small
\centering
\caption{Parameters for different examples}
\begin{tabular} {*5c}
\toprule
& \textbf {Ex~1} &  \textbf {Ex~2}&  \textbf{Ex~3} &\textbf{Ex~4}\\
\midrule
$\mathbf{X}$ &$[-5,5]^2$ & $[-5,5]^2$ & $[-5,5]^2$ & $[-5,5]^3$ \\
$\mathcal{N}$ & $[-0.1,0.1]^2$ & $[-0.1,0.1]^2$ & $[-0.1,0.1]^2$ & $[-0.1,0.1]^3$ \\
deg($c(\bx)$) & $2$ & $6$ & $3$ & $6$  \\ 
deg($s(\bx)$) & 2& $7$ & $7$ & $6$ \\
$\bPsi(\bx)$ & ${4^{th}}$ order & ${9^{th}}$ order & ${7^{th}}$ order & ${8^{th}}$ order \\
\toprule
\end{tabular}
\label{table:OCP param}
\end{table}

\noindent\textbf{Example 1:} Consider the dynamics of controlled simple nonlinear numerical system:
\begin{align*}
\dot x_1=-x_1+x_2,\;\;
\dot x_2= -0.5(x_1+x_2) + 0.5x_1^2x_2 + x_1 u.
\end{align*}
With this example we use ${\cal L}_2$ cost on the control input. For this example, optimal control and optimal cost can be found by solving the HJB equation~\cite{primbs1996optimality} analytically and are given as below:
\[u^\star(\bx)=-x_1x_2,\;\;V^\star(\bx)=0.5 x_1^2+x_2^2.\]
Next, using the proposed method, we get an optimal control with discount factor $\gamma=0$ as follows.
\[k^\star(\bx) = -1.38x_1x_2 + 0.00005x_1 - 0.00004x_2 - 0.0063\]
By rounding off the coefficients of $k^\star(\bx)$, we see that $u^\star(\bx) = k^\star(\bx)$. The small mismatch in decimal is due to the choice of $h_0(\bx)$, which is unique to our formulation but is absent from the primal formulation of OCP. The simulation results are obtained by solving inequality in the constraints corresponding to the case where $h_0\geq 0$. In this example, we collected $2 \times 10^4$ time-series data points by simulating the system to estimate the Koopman generator. Figure~\ref{E7L2_1} shows the comparison of trajectories simulated from the closed-loop system using the optimal control solutions obtained by the HJB approach (dotted red) and the proposed data-driven convex approach (solid black).  


\begin{figure}[h]
\centering
\includegraphics[width=0.9\linewidth]{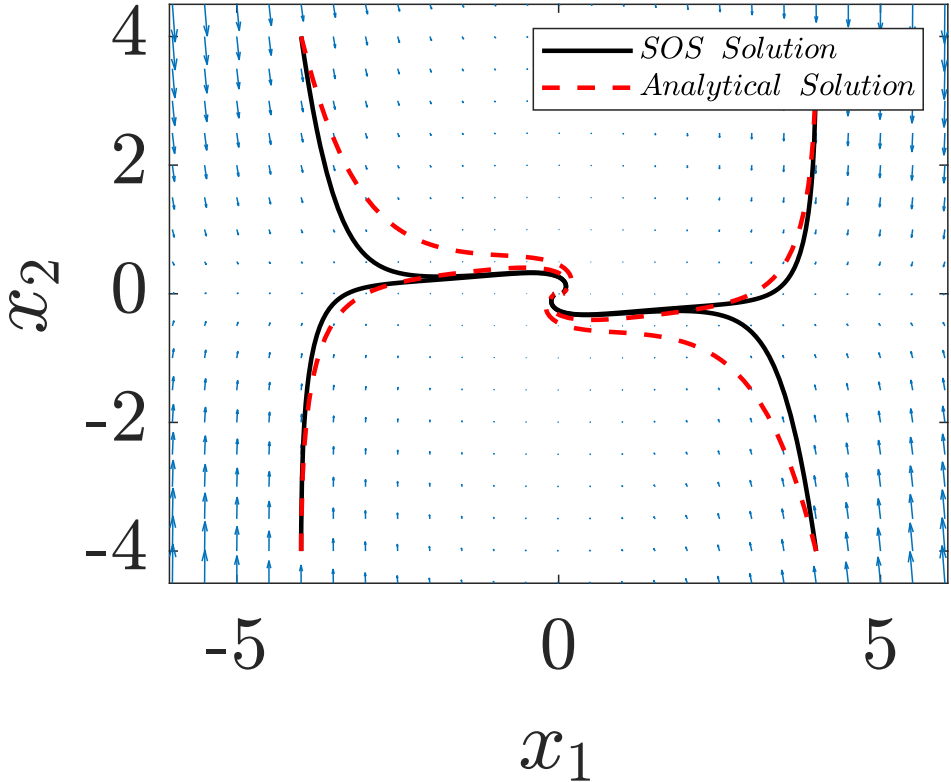}
\caption{$\mathcal{L}_2$ OCP results of system in Example 1, showing converging state trajectories for $\gamma=0$.}
\label{E7L2_1}
\end{figure}

\noindent{\textbf{Example 2.}} Consider the dynamics of controlled Van der Pol oscillator as follows:
\begin{equation*}
\dot{x}_1 = x_2,\,\, \dot{x}_2 = (1-x_1^2)x_2 - x_1 + u.
\end{equation*}
In this example, we solve the ${\cal L}_2$ OCP for different values of the discount factor. Total $2 \times 10^4$ time-series data points are collected from repeated simulations to estimate the Koopman generator. Fig.\ref{E2L2_1} and~\ref{E2L2_6} show the trajectories of the closed-loop system starting from arbitrary initial points obtained from discount factor values, $\gamma=0$ and $\gamma=1$, respectively. We notice that the controller becomes more aggressive for larger $\gamma$  and trajectories converge to the origin faster. This is expected as the OCPs achieve optimal control solutions at an exponential rate for $\gamma > 0$ for which the closed-loop system converges faster than uniform stability.
On the other hand, we observe that, for negative discount factor $\gamma < 0$, the control solution is not guaranteed to stabilize the system to the origin, and the closed-loop dynamics converge to a limit cycle as shown in Fig.~\eqref{E2L2_6} resulting from $\gamma = -5$. This is again expected as the cost function is decreasing exponentially, and even without the stabilizing feedback controller, the optimal cost function is finite.

\begin{figure}[h]
  \centering
  \includegraphics[width=.9\linewidth]{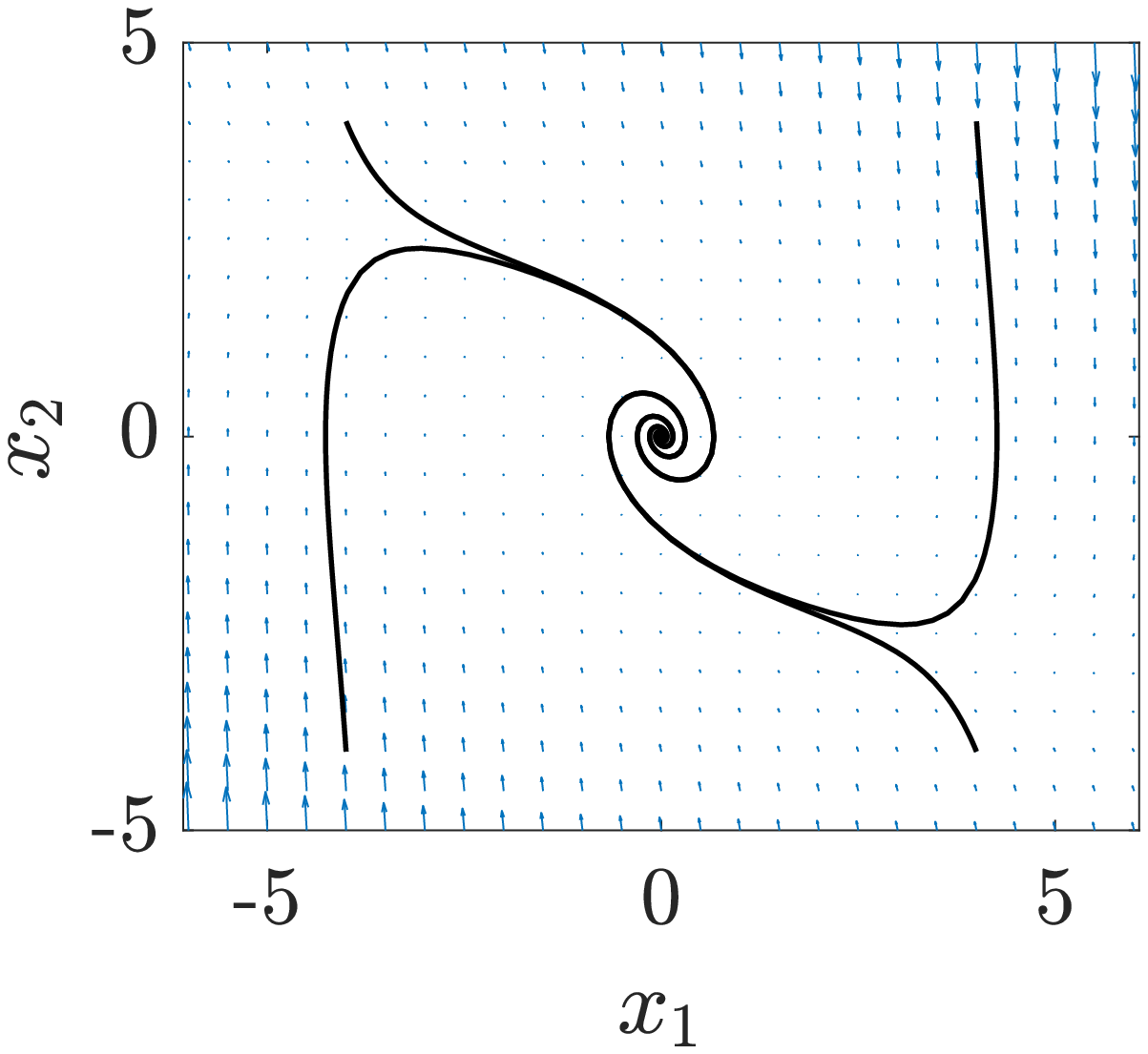}
\caption{Trajectories of closed-loop system with $\mathcal{L}_2$ OCP solution for $\gamma=0$ for Van der pol Oscillator.}
\label{E2L2_1}
\end{figure}

\begin{figure}[h]
  \centering
  \includegraphics[width=.9\linewidth]{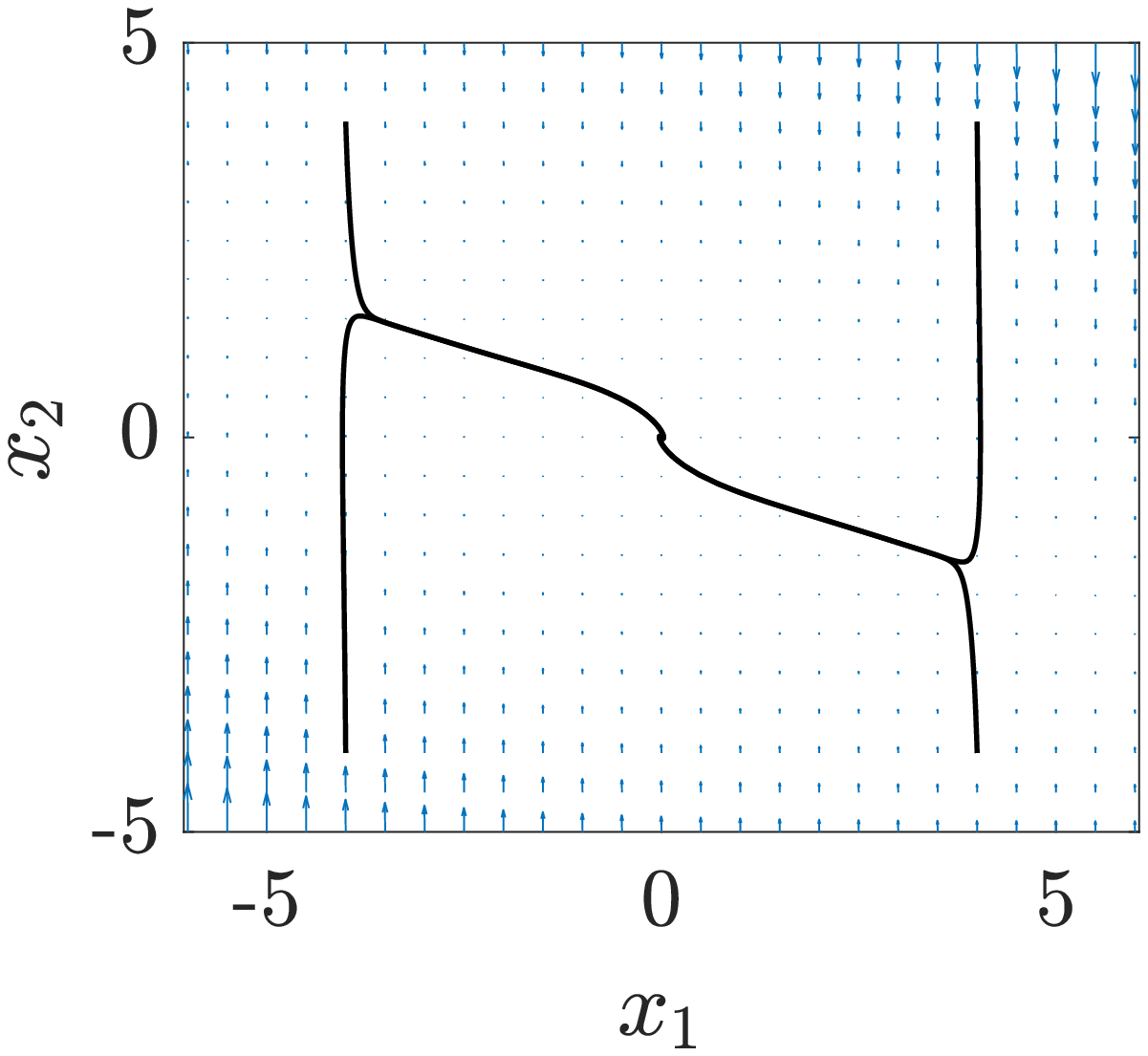}
\caption{Trajectories of closed-loop system with $\mathcal{L}_2$ OCP solution for $\gamma=1$ for Van der pol Oscillator.}
\label{E2L2_2}
\end{figure}




\begin{figure}[h]
  \centering
  \includegraphics[width=.9\linewidth]{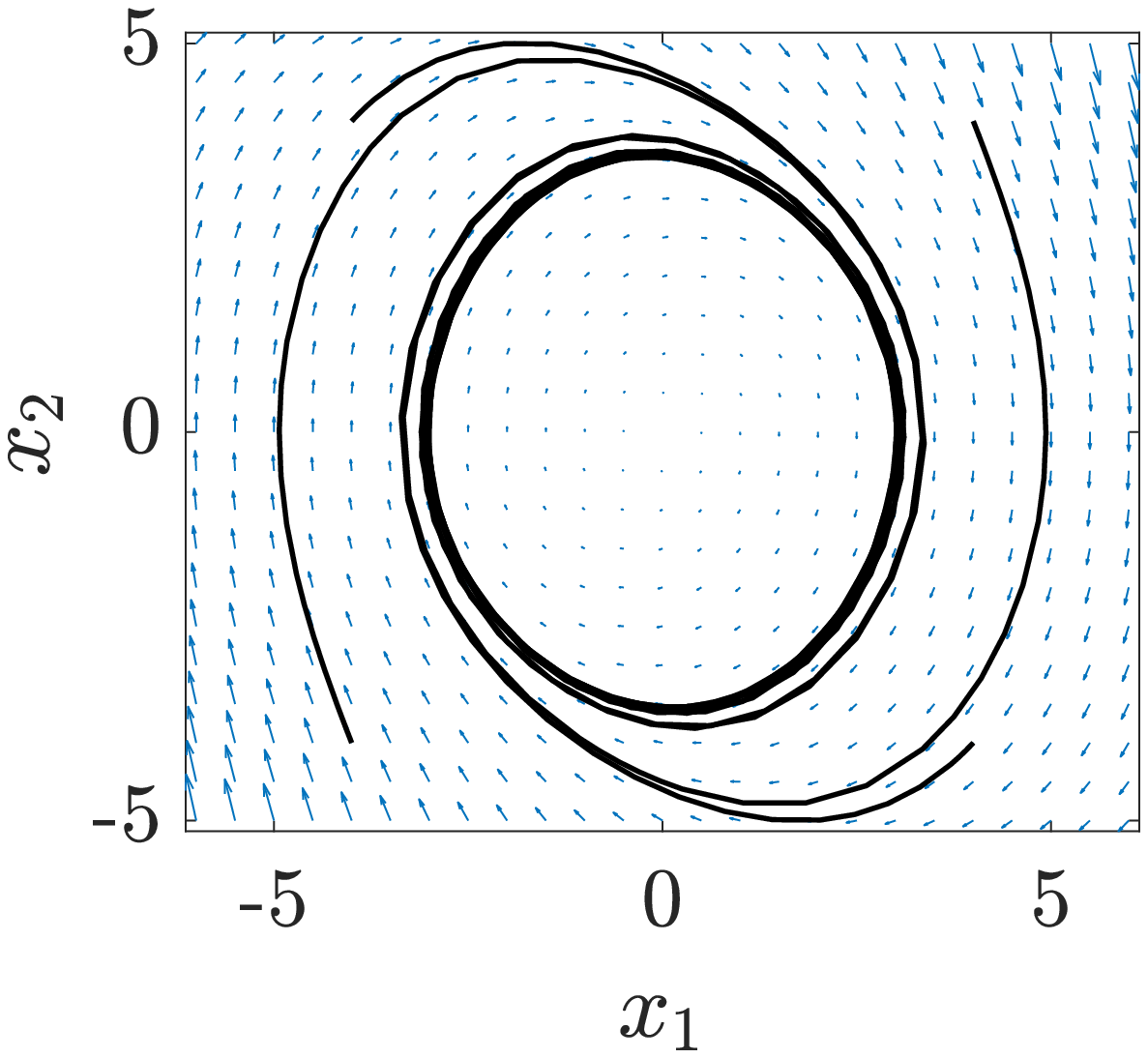}
\caption{Trajectories of closed-loop system with $\mathcal{L}_2$ OCP solution for $\gamma=-5$ for Van der pol Oscillator.}
\label{E2L2_6}
\end{figure}

\noindent\textbf{Example 3:} Consider the dynamics of controlled simple inverted pendulum involing nonpolynomial dynamics:
\begin{equation*}
\dot x_1=x_2,\;\;\dot x_2=-\sin x_1 - 0.2 x_2 + u
\end{equation*}
The number of data points used in the estimation of the Koopman operator equals $2 \times 10^4$. 
The simulation results for the ${\cal L}_2$ optimal control for different discount factor $\gamma$ are shown in Fig.\ref{E5L2_1} - Fig.\ref{E5L2_5}. Similar to Example 2, we notice that, for zero and positive discount factor $\gamma$, the controller obtained by positive discount factor can stabilize the origin at a faster rate than the case of $\gamma = 0$ whereas, for negative discount factor, the origin is not stabilized for the closed-loop system.

\noindent\textbf{Example 4:} Consider the  controlled Lorentz attractor:
\begin{equation*}
\dot x_1=\sigma(x_2-x_1),\dot x_2=x_1(\rho-x_3) - x_2+u,\dot x_3=x_1x_2 - \eta x_3,
\end{equation*}
where $\sigma=10$, $\rho = 28$, and $\eta = \frac{8}{3}$. The open-loop dynamics of the Lorentz system with the above parameter values are chaotic. For this example, we provide simulation results with ${\cal L}_1$ optimal control with $\gamma=0$. We notice that the optimal control can stabilize the system.

\begin{figure}[h]
  \centering
  \includegraphics[width=0.9\linewidth]{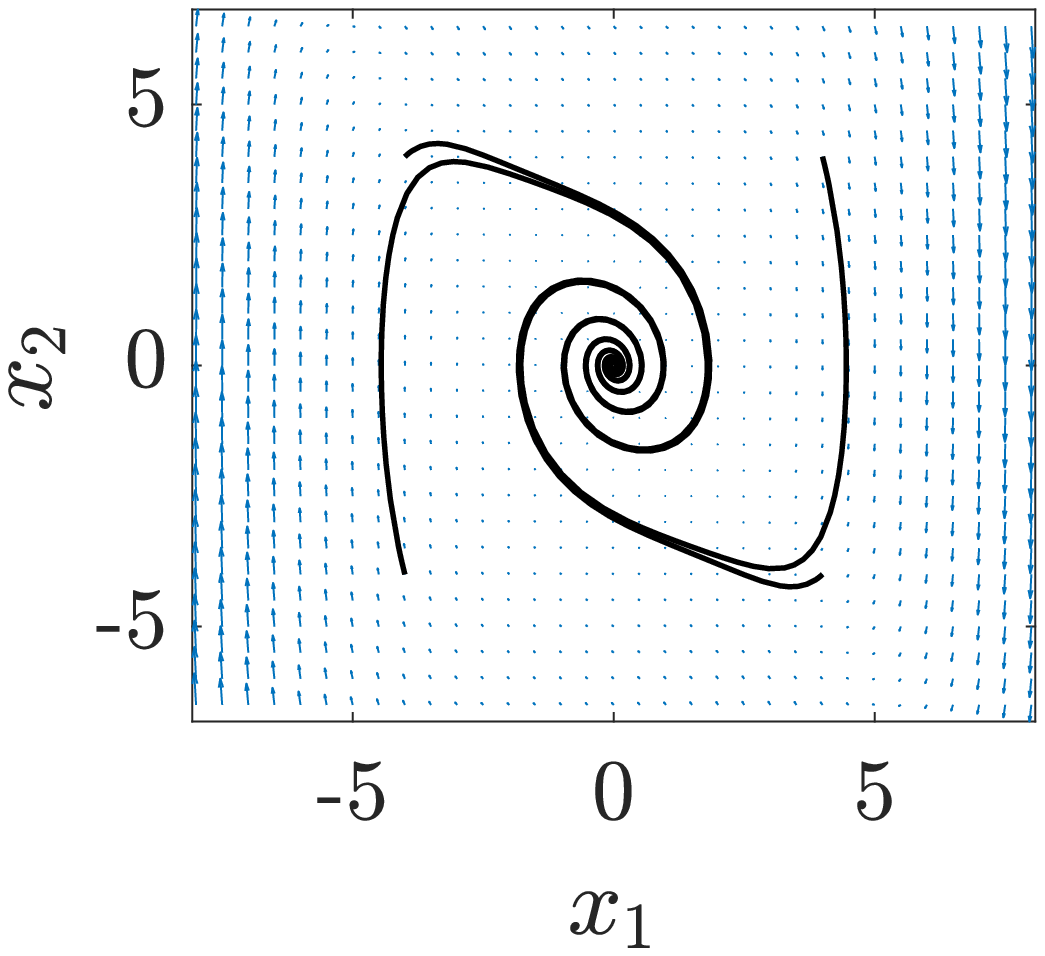}
\caption{$\mathcal{L}_2$ OCP results of simple inverted pendulum, showing converging state trajectories for $\gamma=0$.}
\label{E5L2_1}
\end{figure}

\begin{figure}[h]
  \centering
  \includegraphics[width=0.9\linewidth]{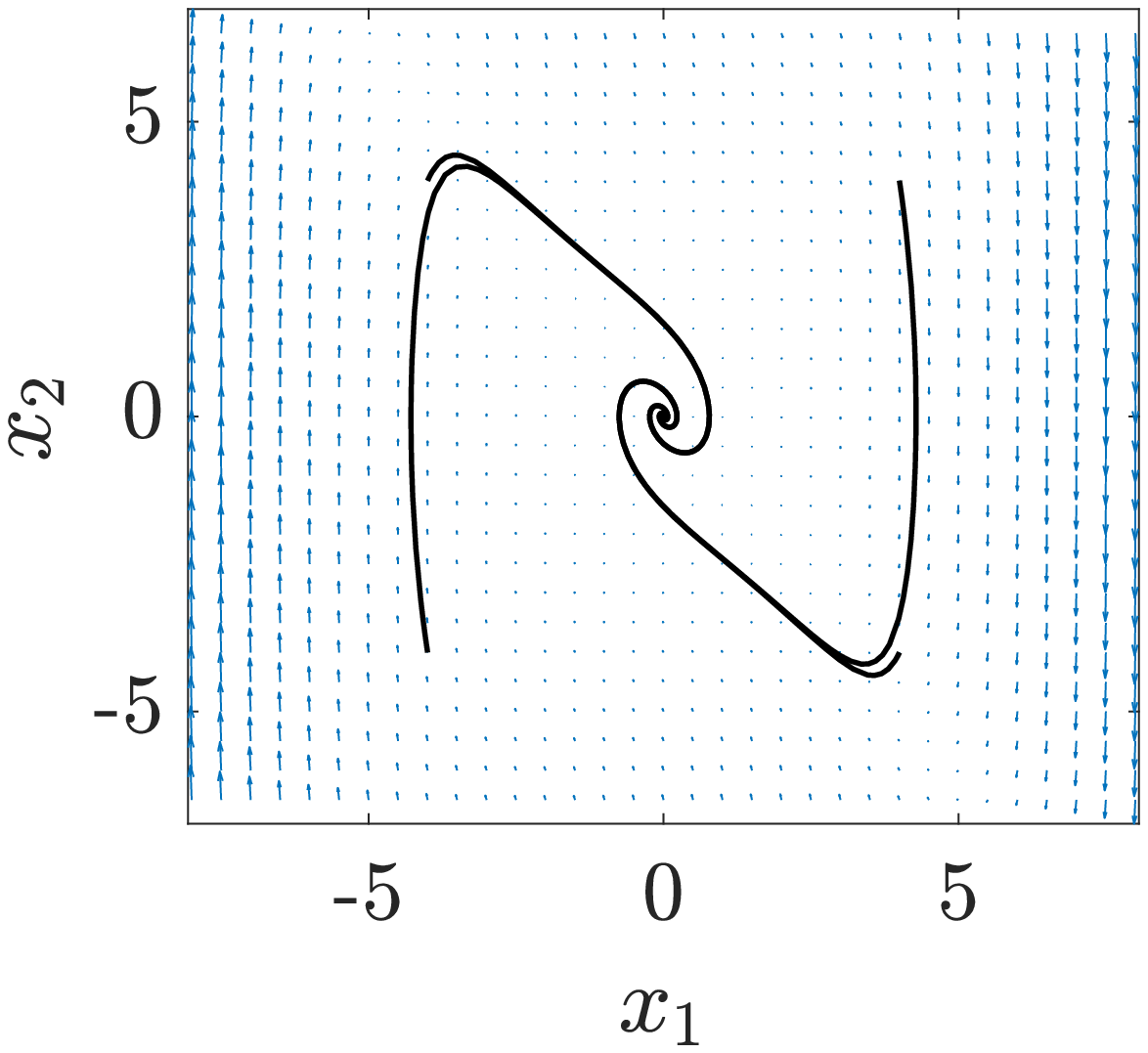}
\caption{$\mathcal{L}_2$ OCP results of simple inverted pendulum, showing converging state trajectories for $\gamma=2$.}
\label{E5L2_2}
\end{figure}



\begin{figure}[h]
  \centering
  \includegraphics[width=0.9\linewidth]{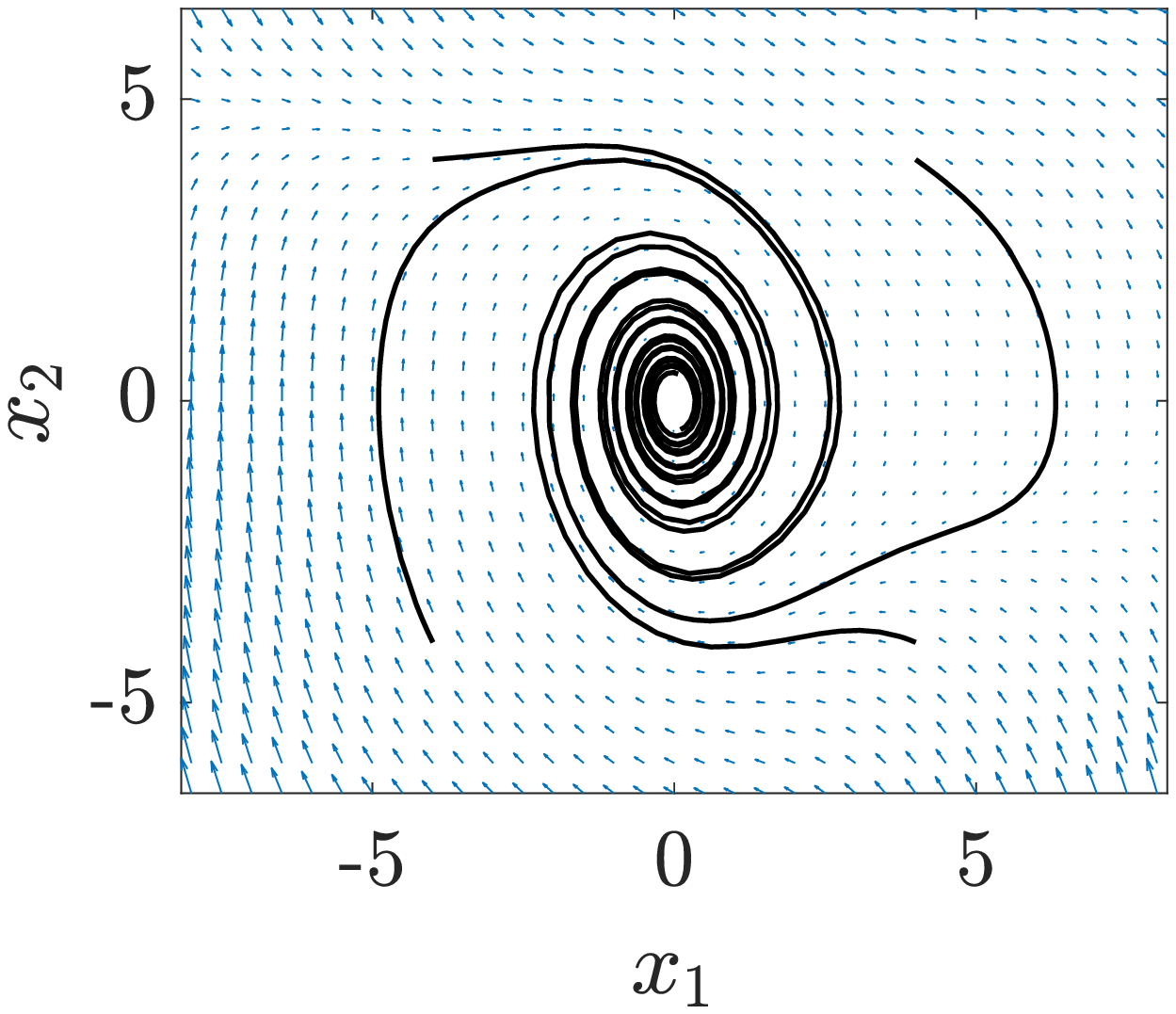}
\caption{$\mathcal{L}_2$ OCP results of simple inverted pendulum, showing converging state trajectories for $\gamma=-5$.}
\label{E5L2_5}
\end{figure}

\begin{figure}[h]
  \centering
  \includegraphics[width=1\linewidth]{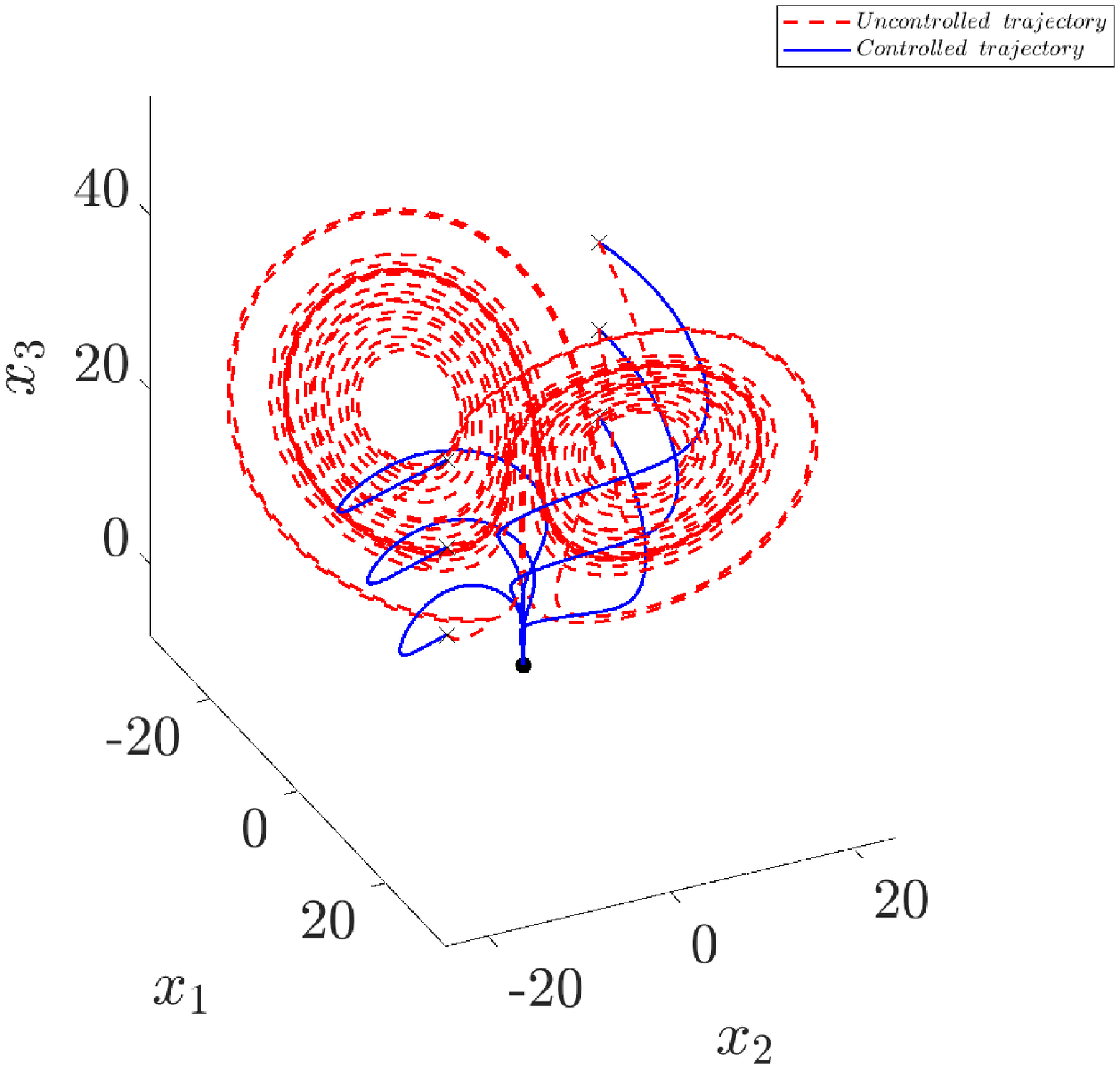}
\caption{$\mathcal{L}_1$ OCP results of Lorentz attractor, showing converging state trajectories for $\gamma=0$.}
\label{E4L1_1}
\end{figure}







\section{Conclusion}\label{sec:conclusion}
A systematic convex optimization-based framework is provided for optimal control of nonlinear systems with a discounted cost function. In contrast to the existing literature of OCP with discounted cost, we consider the OCP problem with both positive and negative discount factor and provide a condition for the existence of optimal control. The OCP is formulated in the dual space of density function as an infinite-dimensional convex optimization problem. A new data-driven algorithm is provided for the computation of optimal control combining methods from Sum-of-Squares optimization and data-driven approximation of linear transfer operators. Simulation results are presented to verify the developed framework on data-driven optimal control.


\section*{Appendix}

\textbf{ Proof of Theorem \ref{theorem_maingeometric}}\\
Consider the feedback control system
\[\dot \bx={\bf f}(\bx)+\bg(\bx)\bk(\bx)\]
and let $\mP^c_t$ and $\mU^c_t$ be the P-F and Koopman operator for the feedback control system. Using the definition of the Koopman operator, the cost in (\ref{ocp_main_discounted}) can be written as 
\[
J(\mu_0)=\int_{\bX_1}\int_0^\infty e^{\gamma t}[\mU_t^c (q+\bk^\top \bR\bk)](\bx)dt h_0(\bx)d\bx.
\]
Using the duality and linearity of the Koopman and P-F operators, we obtain
\[
J(\mu_0)=\int_{\bX_1}\int_0^\infty (q+\bk^\top \bR\bk)(\bx)e^{\gamma t}[\mP_t^c h_0](\bx)dt d\bx.
\]
Defining 
\begin{equation}\rho(\bx):=\int_0^\infty e^{\gamma t}[\mP_t^c h_0](\bx) dt, \label{definingrho}
\end{equation} 
the $J(\mu_0)$ can be written as 
\begin{equation}
J(\mu_0)=\int_{\bX_1} (q(\bx)+\bk(\bx)^\top \bR\bk(\bx)) \rho(\bx)d\bx.\label{ocp_costdiscountproof}
\end{equation}
Defining $\bar \brho(\bx):=\bk(\bx)\rho(\bx)$, the cost function can be written in the form given in \eqref{eqn_ocpdiscount_L2}.
We next show that $\rho(\bx)$ and $\bar \brho$ satisfies the constraints in \eqref{eqn_ocpdiscount_L2}. Following Assumptions \ref{assumption_onocp}, we know that the state cost $q$ is uniformly bounded away from zero in $\bX_1$ and optimal cost function is finite and hence we have
\begin{equation}\infty>J(\mu_0)\geq \int_{\bX_1}q(\bx)\rho(\bx)d\bx\geq c\int_{\bX_1}\rho(\bx)d\bx\label{ss}
\end{equation}
where $c$ is the lower bound for the state cost function $q(\bx)$ on $\bX_1$. The above proves that there exists a constant $M$ such that 
\begin{eqnarray}
\int_{\bX_1}\rho(\bx)d\bx=\int_{\bX_1}\int_0^\infty e^{\gamma t}[\mP_t^c h_0](\bx)dt d\bx\leq M.\label{before_claim}
\end{eqnarray}
We next claim that
\begin{eqnarray}
\lim_{t\to \infty}e^{\gamma t}[\mP_t h](\bx)=0\label{claim}
\end{eqnarray}
for $\mu_0$ almost all $\bx\in \bX_1$. Using Barbalat Lemma, we know that for $f(t)\in {\cal C}^1$, and $\lim_{t\to \infty} f(t)=\alpha$. If $f'(t)$ is uniformly continuous, then $\lim_{t\to \infty} f'(t)=0$.\\
Letting $f(t)=\int_0^t e^{\gamma \tau}\int_{\bX_1}[\mP^c_\tau h_0](\bx)d\bx d\tau$ and hence $f'(t)=e^{\gamma t} \int_{\bX_1}[\mP_t h_0](\bx)d\bx$. By Barbalat Lemma, since $e^{\gamma t}\int_{\bX_1}[\mP^c_t h_0](\bx)d\bx$  is uniformly continuous w.r.t. time for  $\gamma\leq 0$. The uniformly continuity follows from the definition of the P-F semi-group and the fact that solution of dynamical system is uniformly continuous w.r.t. time. We have 
\[0=\lim_{t\to \infty}e^{\gamma t}\int_{\bX_1}[\mP^c_t h_0](\bx)d\bx=\lim_{t\to \infty}\int_{\bX_1}e^{\gamma t}[\mP_t^c h_0](\bx)d\bx\]
which implies 
\begin{equation}\lim_{t\to \infty}e^{\gamma t}[\mP_t^c h_0](\bx)=0
\label{convergence}
\end{equation}
for a.e. $\bx$ w.r.t. $\mu_0$ in $\bX_1$.
Next, we claim that $\rho(\bx)$ as defined in (\ref{definingrho}) can be obtained as a solution of the following equation
\begin{equation}
\nabla\cdot (({\bf f}(\bx)+\bg(\bx)\bk(\bx))\rho(\bx))=\gamma \rho(\bx)+h(\bx),\label{steady_pdegeometric}
\end{equation}
for $x\in \bX_1$.
Substituting (\ref{definingrho}) in (\ref{steady_pdegeometric}), we obtain
\begin{eqnarray}
\nabla\cdot ({\bf f}_c \rho)=\int_0^\infty \nabla\cdot({\bf f}_c (\bx)e^{\gamma t}[\mP^c_t h_0](\bx))dt\nonumber\\
=\int_0^\infty -e^{\gamma t}\frac{d}{dt}[\mP^c_t h_0](\bx)dt\nonumber\\
-e^{\gamma t} [\mP_t^c h_0](\bx)|_{t=0}^\infty+\int_0^\infty \gamma e^{\gamma t} [\mP_t^c h_0](\bx) dt\nonumber\\
=h_0(\bx)+\gamma \rho(\bx)\label{eq22}
\end{eqnarray}
In deriving (\ref{eq22}) we have used the infinitesimal generator property of P-F operator Eq. (\ref{PF_generator}) and the fact that $\lim_{t\to \infty} e^{\gamma t}[\mathbb{P}_t^c h_0](\bx)=0$  following (\ref{convergence}).
Furthermore, since $h_0(\bx)> 0$, it follows that $\rho(\bx)> 0$ from the positivity property of the P-F operator. Combining (\ref{ocp_costdiscountproof}) and (\ref{eq22}) along with the definition of $\bar \rho$, it follows that the OCP problem can be written as  convex optimization problem (\ref{eqn_ocpdiscount}). The optimal control $\bk^\star(\bx)$ obtained as the solution of optimization problem (\ref{eqn_ocpdiscount}) is  a.e. uniform stable (for $\gamma=0$) follows from the results of Theorem \ref{theorem_necc_suffuniform} using the fact that closed loop system satisfies (\ref{steady_pdegeometric}) with $\rho$ that is integrable. The optimal solution  $\rho^\star(\bx)\in {\cal L}_1(\bX_1)\cap {\cal C}^1(\bX_1,\mR_{\geq 0})$ follows from the fact that $h_0\in {\cal L}_1(\mR^n,\mR_{> 0})\cap {\cal C}^1(\mR^n)$  and the definitions of $\rho$ (\ref{definingrho}) and P-F operator (\ref{pf-operator}). \\

\textbf{ Proof of Theorem \ref{theorem_maingeometric_positivediscount}} The proof of this theorem follows exactly along the lines of proof of Theorem \ref{theorem_maingeometric} until equation (\ref{before_claim}). Unlike the proof of Theorem \ref{theorem_maingeometric}, where the claim (\ref{claim}) is proved using Barbalat Lemma, in this proof $\lim_{t\to \infty}[\mP_t h_0](\bx)=0$ for $\mu_0$ almost all $\bx \in \bX_1$ follows from Assumption \ref{assumption_ocppositivediscount}. As $\mu_0(B_t)\leq M e^{-\gamma' t}$ implies $e^{\gamma t}[\mP_t h_0](\bx)\to 0$ for $\mu_0$ a.e. $\bx\in \bX_1$. The rest of the proof then follows again along the lines of proof of Theorem \ref{theorem_maingeometric}. 




\bibliographystyle{apalike}
\bibliography{refs,ref,ref1,autosam}
\end{document}